\documentclass[alphabetic]{amsart}

\usepackage{enumerate, longtable}
\usepackage{amsmath, amscd, amsfonts, amsthm, amssymb, latexsym, comment, stmaryrd, graphicx, dsfont, amsaddr, mathtools, stmaryrd}
\usepackage{xcolor}
\usepackage[all]{xy}
\usepackage{fullpage}
\usepackage{tikz}
\usepackage[
        colorlinks,
        linkcolor=red,  citecolor=blue,
        backref
]{hyperref}

\usepackage{amstext} 
\usepackage{array}
\usepackage{cleveref}

\newtheorem{claim}{Claim}[section]

\newtheorem{thm}{Theorem}

\newtheorem{cor}[claim]{Corollary}
\newtheorem{lem}[claim]{Lemma}

\theoremstyle{definition}

\newtheorem{defi}[claim]{Definition}

\newcommand{\F}{{\mathbb F}}

\newcommand{\Aut}{{\mathrm{Aut}}}
\newcommand{\Out}{{\mathrm{Out}}}
\newcommand{\Inn}{{\mathrm{Inn}}}
\newcommand{\con}{{\mathrm{conj}}}

\newcommand{\Sym}{{\mathrm{Sym}}}

\newcommand{\Hol}{{\mathrm{Hol}}}
\newcommand{\InHol}{{\mathrm{InHol}}}
\newcommand{\inv}{{\mathrm{inv}}}

\newcommand{\CS}{{\mathrm{CS}}}

\newcommand{\isom}{{\xrightarrow{\sim}}}

\newcommand{\proper}{%
  \mathrel{\ooalign{$\lneq$\cr\raise.22ex\hbox{$\lhd$}\cr}}}

\numberwithin{equation}{section}

\begin{document}

\author{Alexei Entin}
\address{Raymond and Beverly Sackler School of Mathematical Sciences, Tel Aviv University, Tel Aviv 69978, Israel}
\email{aentin@tauex.tau.ac.il}
\author{Cindy (Sin Yi) Tsang}
\address{Department of Mathematics, Ochanomizu University, 2-1-1 Otsuka, Bunkyo-ku, Tokyo, Japan}
\email{tsang.sin.yi@ocha.ac.jp}


\title{Normalizer quotients of symmetric groups \\and inner holomorphs}

\begin{abstract} We show that every finite group $T$ is isomorphic to a normalizer quotient $N_{S_n}(H)/H$ for some $n$ and a subgroup $H\leq S_n$. We show that this holds for all large enough $n\ge n_0(T)$ and also with $S_n$ replaced by $A_n$. The two main ingredients in the proof are a recent construction due to Cornulier and Sambale of a finite group $G$ with $\Out(G)\cong T$ (for any given finite group $T$) and the determination of the normalizer in $\Sym(G)$ of the inner holomorph $\InHol(G)\leq\Sym(G)$ for any centerless indecomposable finite group $G$, which may be of independent interest.
\end{abstract}

\maketitle

\section{Introduction}

The following question was raised by M\"uller \cite{Mul12_}: Is every finite group isomorphic to a normalizer quotient of a symmetric group, that is a group of the form $N_{S_n}(H)/H$ for some $n$ and $H\leq S_n$? In the present paper we give a positive answer.

\begin{defi}
    We say that a group $T$ is a \emph{normalizer quotient} of another group $S$, if there exists a subgroup $H\leq S$ such that $N_S(H)/H\cong T$.
\end{defi}

\begin{thm}\label{thm: main} Let $T$ be a finite group. There exists a natural number $n$ such that $T$ is a normalizer quotient of $S_n$.\end{thm}

The question was motivated by a variation of the Inverse Galois Problem, called the Weak Inverse Galois Problem in \cite{DeLe21} and studied also in \cite{FrKo78, Fri80, Tak80, Gey83, LePa18}, in which given a base field $F$ and finite group $T$ one looks for (not necessarily Galois) finite extensions $K/F$ with $\Aut(K/F)\cong T$. Applications of Theorem \ref{thm: main} (and its corollary below) to this problem appear in a separate paper by the first author \cite{Ent24_2}.

\begin{cor}\label{cor: main} Let $T$ be a finite group. For all large enough $n\ge n_0(T)$, the group $T$ is a normalizer quotient of $S_n$ and of $A_n$.\end{cor}

A key ingredient in the proof of Theorem \ref{thm: main} is the determination of the normalizer of the inner holomorph of a centerless indecomposable group, which may be of independent interest. First we recall a few definitions.

Let $G$ be a finite group and let $\Sym(G)$ denote the group of all permutations of the set $G$. The \emph{left regular representation} of $G$ is defined as the map
$$\lambda:G\to\Sym(G):\quad g\mapsto(x\mapsto gx).$$
Similarly, the \emph{right regular representation} of $G$ is defined as the map $$\rho:G\to\Sym(G):\quad g\mapsto(x\mapsto xg^{-1}).$$ 
The \emph{(permutational) holomorph} of $G$ is defined to be
$$\Hol(G)=\lambda(G)\rtimes\Aut(G)=\rho(G)\rtimes\Aut(G)\leq\Sym(G).$$ 
Alternatively, it is not difficult (see \cite[Proposition 7.2]{Childsbook}) to show that
$$\Hol(G) = N_{\Sym(G)}(\lambda(G)) = N_{\Sym(G)}(\rho(G)).$$
In view of this, the \emph{multiple holomorph} of $G$ is defined to be 
\[ \mathrm{NHol}(G) = N_{\Sym(G)}(\Hol(G))=N_{\Sym(G)}(N_{\Sym(G)}(\lambda(G)))=N_{\Sym(G)}(N_{\Sym(G)}(\rho(G))).\]
It is well-known that the quotient $\mathrm{NHol}(G)/\Hol(G)$ acts regularly on the set of regular normal subgroups of $\Hol(G)$ that are isomorphic to $G$ (see \cite[\S 1]{Kohl} for instance) and its structure has been studied for various families of groups $G$ \cite{Kohl,fgabelian,Perfect,largeTG,squarefree}. Here we will restrict to the group $\Inn(G)$ of inner automorphisms of $G$ and consider the normalizer of the \emph{inner holomorph} of $G$, the latter defined to be
$$\InHol(G)=\lambda(G)\rho(G)=\lambda(G)\rtimes\Inn(G)=\rho(G)\rtimes\Inn(G)\leq\Sym(G).$$
Clearly $\Hol(G)$ normalizes $\InHol(G)$. The inversion map $\inv_G\in\Sym(G)$ defined by $\inv_G(g)=g^{-1}$ for any $g\in G$ also normalizes $\InHol(G)$ because it swaps $\lambda(G)$ and $\rho(G)$. Since $\inv_G$ centralizes $\Aut(G)$, it follows that $\inv_G$ normalizes $\Hol(G)$ and we have
$$ N_{\Sym(G)}(\InHol(G)) \geq \langle \Hol(G),\inv_G\rangle = \Hol(G)\cup \Hol(G)\inv_G.$$
We note that $\inv_G\in \Hol(G)$ if and only if $G$ is abelian. In the case that $G$ is centerless (that is $Z(G)=1$), we will give a complete characterization of when equality holds, i.e. when $N_{\Sym(G)}(\InHol(G))$ is as small as possible.

\begin{thm}\label{thm: inhol} Let $G$ be a finite centerless group. We have $N_{\Sym(G)}(\InHol(G))=\langle \Hol(G),\inv_G\rangle$ if and only if $G$ is indecomposable (that is, cannot be written as a direct product of proper normal subgroups).
\end{thm}

Here is a brief sketch of the proof 
of Theorem \ref{thm: main}. Let $T$ be a finite group of order $|T|>2$ and let $p>|T|$ be a prime. The 
Cornulier-Sambale construction (see \S\ref{sec: cs}) 
produces a finite group $G=\CS(T,p)$ 
that has outer automorphism group 
$\Out(G)\cong T$. Furthermore we show 
that $G$ is centerless and 
indecomposable when $p>|T|+1$, so Theorem \ref{thm: 
inhol} applies and we have $N_{\Sym(G)}
(\InHol(G))=\langle 
\Hol(G),\inv_G\rangle$. Using this and some 
further properties of $G$ to be established in 
$\S\ref{sec: cs}$, we will show that 
$H=\langle\InHol(G),\inv_G\rangle$ has the same 
normalizer $N_{\Sym(G)}
(H)=\langle\Hol(G),\inv_G\rangle$ and then deduce that $N_{\Sym(G)}(H)/H\cong\Out(G)\cong T$.

The paper is organized as follows: In $\S\ref{sec: cs}$ we review the Cornulier-Sambale construction and prove some useful properties of it. In $\S\ref{sec: inhol}$ we give the proof of Theorem \ref{thm: inhol}. Finally, the proofs of Theorem \ref{thm: main} and Corollary \ref{cor: main} will be given in \S\ref{sec: proof main}.

{\bf Acknowledgment.} The first author was partially supported by Israel Science Foundation grant no. 2507/19.

\section{The Cornulier-Sambale construction}
\label{sec: cs}

In the present section we summarize a construction due to Cornulier \cite{Cor20_} and Sambale \cite{Sam24_1}, which given a finite group $T$ produces a finite group $G$ with $\Out(G)\cong T$. We will also prove some properties of this construction which will be important for our application.
Cornulier's original construction was cast in the language of Lie algebras using the Lazar-Mal'cev correspondence \cite[\S 10]{Khu98}. Sambale described a similar construction in a purely group-theoretic language. Here we follow Sambale \cite[\S 3]{Sam24_1}.

In what follows let $p$ be a prime and $n$ a natural number. Let $F_n$ be the free group of rank $n$ and consider the group $F$ with presentation $$F=\langle t_1,\ldots,t_n\,|\,w(t_1,\ldots,t_n)^p=1\,:\,w\in F_n\rangle.$$ This is the free exponent $p$ group on generators $t_1,\ldots,t_n$. Consider its lower central series
$$F^{[1]}=F,\quad F^{[l+1]}=[F,F^{[l]}].$$
For any $k$, the group $F/F^{[k+1]}$ is the universal exponent $p$ group of nilpotency class $k$ on $n$ generators, in the sense that if $P_k$ is a nilpotent group of exponent $p$ and nilpotency class $\le k$, then for any elements $x_1,\ldots,x_n\in P_k$ there exists a unique homomorphism $F/F^{[k+1]}\to P_k$ taking $t_i\bmod F^{[k+1]}$ to $x_i$ \cite[p. 6]{Sam24_1}. Moreover $F/F^{[k+1]}$ is a finite $p$-group because each $F^{[l]}/F^{[l+1]}$ is a finite elementary abelian $p$-group \cite[p. 6]{Sam24_1}. We define $l$-fold commutators recursively by setting
$$[g]=g,\quad [g_1,\ldots,g_l]=[g_1,[g_2,\ldots,g_l]].$$
Let us also denote $Q=(\F_p^\times)^n\simeq C_{p-1}^n$.

The group $Q=(\F_p^\times)^n$ acts by automorphisms on $F$ \cite[p. 8]{Sam24_1}. We use a right exponential notation for this action. The action of $a=(a_1,\ldots,a_n)\in Q$ on the generators is given by $t_i^a=t_i^{a_i}$ (raising to the $a_i$-th power). This action descends to $F/F^{[k+1]}$ and satisfies
\begin{equation}\label{eq: action of Q on commutator}[t_{i_1},\ldots,t_{i_k}]^a\equiv [t_{i_1},\ldots,t_{i_k}]^{a_{i_1}\cdots a_{i_k}}\pmod{ F^{[k+1]}}\end{equation}
by \cite[Lemma 6]{Sam24_1}.
Consequently, the action of $Q$ also descends to the quotient group $F/N_kF^{[k+1]}$ where $N_k\leq F^{[k]}$ is any subgroup generated by elements of the form $[t_{i_1},\ldots,t_{i_k}]$ (it follows from Lemma \ref{lem: FkFk1}(i) below that $N_kF^{[k+1]}\trianglelefteq F^{[k]}$ and hence $N_kF^{[k+1]}\trianglelefteq F$ because $F^{[k]}$ is characteristic in $F$).


Now assume that the set of generators $T=\{t_1,\ldots,t_n\}$ is equipped with a group structure with the underlying binary operation denoted by $*$. The following construction appears in \cite[\S 3]{Sam24_1}.

\begin{defi}\label{def: cs} The \emph{Cornulier-Sambale group} of $T$ with respect to the prime $p$ is defined as follows: first set $P=F/NF^{[n+1]}$, where 
$$N=\langle [t*t_1,t*t_2,\ldots,t*t_{n-1},t*t_1]\,:\,t\in T\rangle.$$ As noted above $Q=(\F_p^\times)^n$ acts on $P$ and we define $\CS(T,p)=P\rtimes Q$ with respect to this action.\end{defi}

The main property of $\CS(T,p)$ is the following:

\begin{thm}[{Cornulier-Sambale \cite[Theorem 8]{Sam24_1}}] \label{thm: cs}$\Out(\CS(T,p))\cong T$ whenever $p>|T|>2$.\end{thm}

Next we want to show that whenever $p>|T|+1>2$, the group $G=\CS(T,p)$ is centerless (i.e. $Z(G)=1$) and indecomposable (i.e. one cannot write $G=H\times K,\,1\proper H,K\proper G$). This will be needed in order to apply Theorem \ref{thm: inhol} later. First we need a couple of lemmas.

\begin{lem}\label{lem: FkFk1} The following statements hold.
\begin{enumerate}[(i)]
\item[(i)]$F^{[k]}/F^{[k+1]}$ is a finite elementary abelian $p$-group, with a basis consisting of (not necessarily all) elements of the form $[t_{i_1},\ldots,t_{i_k}]F^{[k+1]}$. If $k=1$ the basis is precisely $\{t_iF^{[2]}:\,1\le i\le n$\}.
\item[(ii)] The group $Q$ acts diagonally with respect to this basis.
\item[(iii)] If $p>k+1$ no basis element is fixed by $Q$.
\end{enumerate}
\end{lem}

\begin{proof} For (i) see \cite[p. 6]{Sam24_1} (the second assertion of (i) is easily seen by applying the universality of $F/F^{[2]}$ as an abelian group of exponent $p$ to the group $\mathbb{F}_p^n$). Part (ii) follows from (\ref{eq: action of Q on commutator}). For (iii) use (\ref{eq: action of Q on commutator}) to conclude that the action of $Q$ on a basis element $[t_{i_1},\ldots,t_{i_k}]F^{[k+1]}$ is trivial if and only if each generator is repeated a multiple of $p-1$ times in the above commutator, which is impossible if $p>k+1$.\end{proof}

\begin{lem}\label{lem: centralizer semidirect} Let $M=A\rtimes B$ be an internal semidirect product of groups. Then
$$Z(M)=\left\{(a,b):\,b\in Z(B),\,a^B=\{a\},\, \con(ab)|_A=\mathrm{id}_A\right\}.$$ Here $\con(ab)\in\Inn(A\rtimes B)$ denotes conjugation by $ab$ and $a^B$ is the orbit of $a$ under conjugation by $B$.
\end{lem}

\begin{proof} Every element in $M$ can be written uniquely as $ab$ for $a\in A, b\in B$. The condition $ab\in Z(M)$ is equivalent to $aba'=a'ab, \, abb'=b'ab$ for all $a'\in A, b'\in B$. These can be rewritten as
$$a(ba'b^{-1})b=a'ab,\,abb'=(b'ab'^{-1})b'b.$$
Since $A\cap B=1$, the above are equivalent to
$$a'=(ab)^{-1}a'(ab),\,a=b'ab'^{-1},\,bb'=b'b.$$
The validity of these conditions for all $a'\in A,\,b'\in B$ precisely means that $\con(ab)|_A=\mathrm{id}_A,\,a^B=\{a\}$, and $b\in Z(B)$, as desired. \end{proof}

\begin{lem}\label{lem: CS is centerless} $Z(\CS(T,p))=1$ whenever $p>|T|+1>2$.\end{lem}

\begin{proof} Write $G=\CS(T,p)=P\rtimes Q$ as in Definition \ref{def: cs}. By Lemma \ref{lem: centralizer semidirect} it is enough to show that $Q$ has no nontrivial fixed elements in $P$ and that it acts faithfully on $P$. We have a normal series
\begin{equation}\label{eq: P filtration}1\trianglelefteq F^{[n]}/NF^{[n+1]}\trianglelefteq F^{[n-1]}/NF^{[n+1]}\trianglelefteq\ldots\trianglelefteq F^{[1]}/NF^{[n+1]}=P.\end{equation} The action of $Q$ descends to the quotients of this series, which are $$F^{[n]}/NF^{[n+1]},\,F^{[n-1]}/F^{[n]},\,\ldots,F^{[1]}/F^{[2]},$$
and if $Q$ has a nontrivial fixed element in $P$ then it has a nontrivial fixed element in one of these quotients (if $1\neq x\in P$ is fixed by $Q$ look at the first subgroup in (\ref{eq: P filtration}) containing $x$ and the image of $x$ in the corresponding quotient). Moreover, if the action of $Q$ on one of these quotients is faithful, then so is its action on $P$. By Lemma \ref{lem: FkFk1} and the condition $p>|T|+1=n+1$, each of the above quotients is an elementary abelian $p$-group having a basis consisting of elements not fixed by $Q$, and $Q$ acts diagonally with respect to this basis. This implies that no nontrivial element of these quotients is fixed by $Q$ and that the action of $Q$ on $F^{[1]}/F^{[2]}$ is faithful, which concludes the proof.\end{proof}


\begin{lem}\label{lem: CS is indecomposable} $\CS(T,p)$ is indecomposable whenever $p>|T|>2$.\end{lem}

\begin{proof} We use the notation of Definition \ref{def: cs}. Assume by way of contradiction that $\CS(T,p)=P\rtimes Q=H\times K$ for some proper normal subgroups $H, K$. Denote $H'=H\cap P,\,K'=K\cap P$. Since $P$ is the unique $p$-Sylow subgroup of $P\rtimes Q$, we have that $H',K'$ are the unique $p$-Sylow subgroups of $H,K$ respectively and $P=H'K'$.
In what follows we identify $P/(F^{[2]}/NF^{[n+1]})$ with $F/F^{[2]}$ and $P/(F^{[3]}/NF^{[n+1]})$ with $F/F^{[3]}$ in the natural way (recall (\ref{eq: P filtration}) and the assumption $n=|T|>2$). 
The congruences we write down below are well-defined in light of these identifications. 

Let $x\in H'$ and write (as one may by the case $k=1$ of Lemma \ref{lem: FkFk1}(i)) $x\equiv t_{i_1}^{\lambda_1}t_{i_2}^{\lambda_2}\cdots t_{i_r}^{\lambda_r}\pmod{F^{[2]}}$ with $i_1,\ldots,i_r$ distinct and $\lambda_i\in\F_p^\times$. Assume that $yq\in K$ for some $y\in P,\,q\in Q$. Since $K$ centralizes $H$ we have $xyq=yqx=y(qxq^{-1})q$ and therefore $y^{-1}xy=qxq^{-1}=x^{q^{-1}}$. It follows that
$$\prod_{j=1}^rt_{i_j}^{\lambda_j}\equiv x\equiv y^{-1}xy\equiv x^{q^{-1}}\equiv\prod_{j=1}^rt_{i_j}^{\lambda_j/a_{i_j}}\pmod{F^{[2]}},$$ where $q=(a_1,\ldots,a_n)$. 
Using the case $k=1$ of Lemma \ref{lem: FkFk1}(i) we conclude that $a_{i_j}=1,\,1\le j\le r$. Thus for any $(a_1,\ldots,a_n)\in\pi(K)$ ($\pi:P\rtimes Q\to Q$ is the projection to the second factor) we have $a_l=1$ for any index $l$ for which there exists an element $x\in H'$ such that the basis element $t_l\bmod F^{[2]}$ occurs (with a nonzero coefficient) in the expansion of $x\bmod F^{[2]}$ in the basis $t_1\bmod F^{[2]},\,\ldots\,,t_n\bmod F^{[2]}$. Denote the set of such indices $l$ by $L$.

Similarly, for any $(a_1,\ldots,a_n)\in\pi(H)$ we have $a_m=1$ for any index $m$ for which there exists an element $y\in K'$ such that the basis element $t_m\bmod F^{[2]}$ occurs (with a nonzero coefficient) in the expansion of $y\bmod F^{[2]}$ in the basis $t_1\bmod F^{[2]},\,\ldots\,,t_n\bmod F^{[2]}$. Denote the set of such indices $m$ by $M$.

Since $\pi(HK)=Q$ and $H'K'=P$, it follows that $L,M$ are disjoint (if $r\in L\cap M$ were to exist then the $r$-th coordinate of any $q\in \pi(HK)$ would always be $1$), $\{1,\ldots,n\}=L\cup M$ (if $r\not\in L\cup M$ then the natural surjection $P\rightarrow  F/F^{[2]}$ would not contain $t_r\bmod F^{[2]}$ in its image), and
\begin{equation}\label{eq: HK}\pi(H)=\{(a_1,\ldots,a_n):\,a_m=1\,\, \forall m\in M\},\quad\pi(K)=\{(a_1,\ldots,a_n):\,a_l=1\,\, \forall l\in L\},\end{equation}
\begin{equation}\label{eq: HK2} H'\bmod F^{[2]}=\left\{\prod_{l\in L}t_l^{a_l}\bmod F^{[2]}:\,a_l\in\F_p\right\},\quad K'\bmod F^{[2]}=\left\{\prod_{m\in M}t_m^{a_m}\bmod F^{[2]}:\,a_m\in\F_p\right\}.\end{equation}

First assume that $L,M\neq\emptyset$. Let $l\in L,m\in M$ and let $x\in H',y\in K'$ be such that $x\equiv t_l\pmod{F^{[2]}},\,y\equiv t_m\pmod{F^{[2]}}.$ By \cite[Lemma 6]{Sam24_1} we have $[x,y]\equiv [t_l,t_m]\not\equiv 1\pmod {F^{[3]}}$ (the last incongruence follows from the universality of $F/F^{[3]}$ as an exponent $p$ group of nilpotency class $2$ on the generators $t_1,\ldots,t_n$ and the fact that there is a nonabelian group of exponent $p$ and nilpotency class $2$, e.g. the group of unipotent $3\times 3$ upper triangular matrices over $\F_p$). But $H',K'$ centralize each other, so $[x,y]=1$, a contradiction.

Next assume $L=\emptyset$, in which case $M=\{1,\ldots,n\}$. From (\ref{eq: HK}),(\ref{eq: HK2}) we have $\pi(K)=Q$ and the projection $K'\to F/F^{[2]}$ is surjective. Denote by $\overline K'$ the preimage of $K'$ under the quotient map $F\to F/NF^{[n+1]}=P$. The projection $\overline K'\to F/F^{[2]}$ is surjective and it follows by induction on $k$ using \cite[Lemma 6]{Sam24_1} (which implies that $F^{[k-1]}/F^{[k]}$ is generated by the classes of $(k-1)$-fold commutators of elements from $\overline K'$) that the projection $\overline K'\to F/F^{[k]}$ is surjective for every $k$. In particular the projection $\overline K'\to F/F^{[n+1]}$ is surjective, and a fortiori the projection $\overline K'\to P$ is surjective, i.e. $K'=P$. Since $\pi(K)=Q$ we obtain that $K=P\rtimes Q$ is not a proper subgroup, a contradiction.

The case $M=\emptyset$ is handled similarly, so we obtain a contradiction in all cases, establishing the assertion of the lemma.
\end{proof}

To apply Theorem \ref{thm: inhol} to prove Theorem \ref{thm: main}, we will need one more property of the group $G=\CS(T,p)$.

\begin{lem}\label{lem: char 1} Let $M = (A\times A)\rtimes C_2$, where $A$ is a finite abelian group with $|A|>2$ and the action of the nontrivial element of $C_2$ on $A\times A$ is by swapping the coordinates. Then $A\times A$ is a characteristic subgroup of $M$.\end{lem}


\begin{proof}
We regard $A\times A$ and $C_2$ as subgroups of $M$ and let us write elements of $M$ in the form $(a_1,a_2)c$ for $a_1,a_2\in A,\,c\in C_2$. We will show that $A\times A$ is the unique abelian subgroup of $M$ of index 2, which implies the assertion of the lemma. 

Let $B\leq M$ be an abelian subgroup of index 2. Assume by way of contradiction that $B\neq A\times A$. Then $D = B\cap (A\times A)$ is a subgroup of $A\times A$ of index $2$. Let $b\in B\setminus A\times A$. Then $b$ centralizes $D$ (since $B$ is abelian). Write $b=(a_1,a_2)i$ for $a_1,a_2\in A,\,1\neq i\in C_2$.

Since $A\times A$ is abelian, for any $(d_1,d_2)\in D$ we have
$$(d_1,d_2)=b^{-1}(d_1,d_2)b=i^{-1}(a_1,a_2)^{-1}(d_1,d_2)(a_1,a_2)i=i^{-1}(d_1,d_2)i=(d_2,d_1),$$ so $d_1=d_2$. Thus $D$ is contained in the diagonal $\Delta_A\leq A\times A$, which has index $|A|>2$ in $A\times A$. Since $[A\times A:D]=2$, we have arrived at a contradiction. This completes the proof.
\end{proof}



\begin{lem}\label{lem: char 2} Let $G=\CS(T,p)$. Then $\InHol(G)$ is a characteristic subgroup of $\langle\InHol(G),\inv_G\rangle$ whenever $p>|T|+1>2$.\end{lem}

\begin{proof} Since the group $G$ is centerless by Lemma \ref{lem: CS is centerless}, we have $\InHol(G)=\lambda(G)\times\rho(G)$. Now, conjugation by $\inv_G$ swaps $\lambda(G)$ and $\rho(G)$: more precisely, we have $\lambda(g)\inv_G = \inv_G\rho(g)$ for any $g\in G$. This means that $\inv_G$ normalizes $\InHol(G)$ and therefore $\langle \InHol(G),\inv_G\rangle=\InHol(G)\rtimes\langle\inv_G\rangle$. Since $\inv_G$ is an involution, in particular $[\langle \InHol(G),\inv_G\rangle:\InHol(G)]=2$. 

We may write $G=P\rtimes Q$ where $P$ is the unique (since it is normal) $p$-Sylow subgroup of $G$ and $Q\cong C_{p-1}^{|T|}$. Then $\lambda(P)\rho(P)$ is the unique  $p$-Sylow subgroup of $\langle\InHol(G),\inv_G\rangle$ (it is normal because conjugation by $\inv_G$ swaps $\lambda(P)$ and $\rho(P)$) and is therefore a characteristic subgroup of $\langle\InHol(G),\inv_G\rangle$. Thus it remains to show that $\InHol(G)/\lambda(P)\rho(P)$ is characteristic in $\langle\InHol(G),\inv_G\rangle/\lambda(P)\rho(P)\isom (Q\times Q)\rtimes C_2$ (the action of $1\neq i\in C_2$ is by swapping the coordinates). But under the above isomorphism $\InHol(G)/\lambda(P)\rho(P)$ corresponds to $Q\times Q$, which is characteristic in $(Q\times Q)\rtimes C_2$ by Lemma \ref{lem: char 1}. This completes the proof.\end{proof}

\section{Normalizer of the inner holomorph}
\label{sec: inhol}

In what follows let $G$ be a finite group. A subgroup $R\le \Sym(G)$ is said to be \textit{regular} if its natural action on $G$ is regular, or equivalently, if the map
$$ R\to G:\quad \sigma\mapsto \sigma(1_G)$$
is bijective. Regular subgroups of $\Sym(G)$ come in pairs in some sense because if $R$ is a regular subgroup of $\Sym(G)$, then so is its centralizer $C_{\Sym(G)}(R)$. Moreover $C_{\Sym(G)}(R)\cong R$ and $C_{\Sym(G)}(C_{\Sym(G)}(R)) = R$. Of course, if $R$ is abelian, then $C_{\Sym(G)}(R) =R$ and we do not have a genuine pair of regular subgroups. All of these facts are easy to verify or one can see \cite[\S3]{Kohl4p}. For example, $\lambda(G)$ and $\rho(G)$ are regular subgroups of $\Sym(G)$ isomorphic to $G$, and they are centralizers of each other. It is well-known that isomorphic regular subgroups are conjugates in $\Sym(G)$ (a proof can be found in \cite[Lemma 2.1]{CT}). In particular, the regular subgroups of $\Sym(G)$ that are isomorphic to $G$ are exactly the conjugates of $\lambda(G)$. For example, $\lambda(G)$ and $\rho(G)$ are conjugates via the inversion map $\inv_G$.

With the above observations, we can give a characterization of when the equality $N_{\Sym(G)}(\InHol(G)) = \langle \Hol(G),\inv_G\rangle$ holds in terms of regular subgroups, as follows.

\begin{lem}\label{prop:characterization} We have $N_{\Sym(G)}(\InHol(G)) = \langle \Hol(G),\inv_G\rangle$ if and only if $\lambda(G)$ and $\rho(G)$ are the only regular subgroups $R\le\InHol(G)$ isomorphic to $G$ for which $C_{\Sym(G)}(R)\leq \InHol(G)$.
\end{lem}
\begin{proof} Since $\InHol(G) = \lambda(G)\rho(G)$, for any $\pi \in \Sym(G)$ it is clear that
\begin{align*}
    \pi \in N_{\Sym(G)}(\InHol(G)) & \iff \pi^{-1} \lambda(G)\pi,\pi^{-1}\rho(G)\pi\leq \InHol(G)\\
    & \iff \pi^{-1}\lambda(G)\pi,C_{\Sym(G)}(\pi^{-1}\lambda(G)\pi)\leq \InHol(G).
\end{align*}
Moreover, since $\Hol(G) = N_{\Sym(G)}(\lambda(G))$ and $\inv_G$ swaps $\inv_G^{-1} \lambda(G)\inv_G= \rho(G)$, we have
\begin{align*}
    \pi^{-1} \lambda(G)\pi = \lambda(G) &\iff \pi \in \Hol(G),\\
    \pi^{-1}\lambda(G)\pi = \rho(G)&\iff \pi \in \Hol(G)\inv_G.
\end{align*}
We now deduce the assertion of the lemma because the regular subgroups of $\Sym(G)$ isomorphic to $G$ are precisely the conjugates of $\lambda(G)$.
\end{proof}

In the case that $G$ is centerless, the product
$$\InHol(G) = \lambda(G)\rho(G) = \lambda(G)\times \rho(G)$$
is direct, and the regular subgroups of $\InHol(G)$ (not necessarily isomorphic to $G$) may be parametrized in terms of these so-called fixed point free pairs of homomorphisms. We will not need it but let us remark that there is an extension of this result to all regular subgroups of $\Hol(G)$ in \cite[Proposition 2.5]{CT1}. Below let $N$ denote a finite group of the same order as $G$.

\begin{defi}
    A pair $f,g : N \to G$ of homomorphisms is said to be \textit{fixed point free} if $f(x) = g(x)$ has no solutions other than $x = 1_N$. 
\end{defi}

\begin{lem}\label{lem:regular}
For any fixed point free pair $f,g : N \to G$ of homomorphisms, the set
$$R_{(f,g)} = \{\lambda(f(x))\rho(g(x)) : x\in N\}$$
is a regular subgroup of $\InHol(G)$ isomorphic to $N$. In the case that $G$ is centerless, every regular subgroup of $\InHol(G)$ isomorphic to $N$ arises in this way.
\end{lem}
\begin{proof}
See \cite[\S2 and Proposition 6]{ByottChilds}.
\end{proof}

Here we are interested in the regular subgroups of $\InHol(G)$ whose centralizer also lies inside $\InHol(G)$.

\begin{lem}\label{lem:centralizer} Let $f,g:N\to G$ and $f',g':N\to G$ be two fixed point free pairs of homomorphisms. In the case that $G$ is centerless, we have $R_{(f',g')} = C_{\Sym(G)}(R_{(f,g)})$ if and only if $[f(N),f'(N)]=[g(N),g'(N)]=1$.
\end{lem}
\begin{proof} For any $x,y\in N$, note that
\begin{align*}
\lambda(f(x))\rho(g(x))\cdot\lambda(f'(y))\rho(g'(y))&=\lambda(f(x)f'(y))\rho(g(x)g'(y))\\
\lambda(f'(y))\rho(g'(y))\cdot \lambda(f(x))\rho(g(x)) & = \lambda(f'(y)f(x)) \rho(g'(y)g(x))
\end{align*}
because $\lambda(G)$ and $\rho(G)$ centralize each other. Since $G$ is centerless, we have $\lambda(G)\cap \rho(G)=1$, and the above elements are equal if and only if $f(x)f'(y) = f'(y)f(x),\, g(x)g'(y) = g'(y)g(x)$. It follows that $R_{(f',g')}$ and $R_{(f,g)}$ are centralizers of each other if and only if $[f(N),f'(N)]=[g(N),g'(N)]=1$.
\end{proof}

\begin{proof}[Proof of Theorem \ref{thm: inhol}]
First suppose that $G$ is decomposable. Then $G = H\times K$ for some proper nontrivial normal subgroups $H$ and $K$. Clearly $f,g:G\to G$ and symmetrically $g,f:G\to G$ defined by
$$f(hk) = h,\quad g(hk) = k\quad (h\in H,k\in K)$$
are fixed point free pairs of homomorphisms. Then $R_{(f,g)}$ and $R_{(g,f)}$ are regular subgroups of $\InHol(G)$ that are isomorphic to $G$ by Lemma \ref{lem:regular}, and they are neither $\lambda(G)$ nor $\rho(G)$ by the assumption on $H,K$. Since $f(G) = H$ and $g(G) = K$ centralize each other, we have $C_{\Sym(G)}(R_{(f,g)})=R_{(g,f)}\leq\InHol(G)$ by Lemma \ref{lem:centralizer}, and so $N_{\Sym(G)}(\InHol(G))$ strictly contains $\langle\Hol(G),\inv_G\rangle$  by Lemma \ref{prop:characterization}.

Next suppose that $N_{\Sym(G)}(\InHol(G))$ strictly contains $\langle\Hol(G),\inv_G\rangle$. Then $\InHol(G)$ contains a regular subgroup $R$ isomorphic to $G$ that is neither $\lambda(G)$ nor $\rho(G)$ for which $C_{\Sym(G)}(R)$ also lies in $\InHol(G)$, by Lemma \ref{prop:characterization}. We also know from Lemma \ref{lem:regular} that $R = R_{(f,g)}$ and $C_{\Sym(G)}(R) = R_{(f',g')}$ for some fixed point free pairs $f,g:G\to G$ and $f',g':G\to G$ of homomorphisms. Moreover, we have
\begin{equation}\label{[]}
[f(G),f'(G)] = [g(G),g'(G)]=1
\end{equation}
by Lemma \ref{lem:centralizer}. It is a consequence of fixed point free-ness that
$$\ker(f)\cap \ker(g) = \ker(f')\cap \ker(g') = 1,$$
and as shown in \cite[Proposition 1]{ByottChilds}, that
$$G =f(G)g(G) = f'(G)g'(G).$$
Since $G$ is centerless, the above and (\ref{[]}) yield that 
$$ f(G)\cap g(G) = f'(G)\cap g'(G) = 1.$$ 
But then we get the decomposition $G = \ker(f)\times \ker(g)$ because 
\begin{align*}
|\ker(f)\ker(g)| & =|\ker(f)||\ker(g)|\\
&=|G|/|f(G)| \cdot |G|/|g(G)|\\
&=|G|^2/|f(G)g(G)|\\
& = |G|.
\end{align*}
The assumption $R\neq\lambda(G),\rho(G)$ implies that $f$ and $g$ are both nontrivial, so $\ker(f)$ and $\ker(g)$ are proper normal subgroups of $G$. It follows that $G=\ker(f)\times\ker(g)
$ is decomposable.
\end{proof}

\section{Derivation of Theorem \ref{thm: main} and Corollary \ref{cor: main}}

\begin{proof}[Proof of Theorem \ref{thm: main}]
\label{sec: proof main}

If $T=1$ we may take $n=1,\,H=S_1$, and if $T\cong C_2$ we may take $n=2,\,H=1$. Hence we may assume $|T|>2$. Pick a prime $p>|T|+1$ and denote $G=\CS(T,p)$. By Theorem \ref{thm: cs} we have $\Out(G)\cong T$. By Lemmas \ref{lem: CS is centerless} and \ref{lem: CS is indecomposable} the group $G$ is centerless and indecomposable, so Theorem \ref{thm: inhol} applies and we have $N_{\Sym(G)}(\InHol(G))=\langle \Hol(G),\inv_G\rangle$. Take $$H=\langle \InHol(G),\inv_G\rangle\leq\Sym(G).$$
We will show that $N_{\Sym(G)}(H)/H\cong T$, from which the assertion of Theorem \ref{thm: main} follows immediately. 

Since $\InHol(G)$ is a characteristic subgroup of $H$ by Lemma \ref{lem: char 2}, we have
$$N_{\Sym(G)}(H)\leq N_{\Sym(G)}(\InHol(G))=\langle \Hol(G),\inv_G\rangle.$$
The reverse inclusion is obvious, so we have equality $N_{\Sym(G)}(H)=\langle\Hol(G),\inv_G\rangle$. Moreover, since $\inv_G$ normalizes $\Hol(G)$ and $\InHol(G)$ we have $$\langle\Hol(G),\inv_G\rangle=\Hol(G)\rtimes\langle \inv_G\rangle,\quad H=\InHol(G)\rtimes\langle\inv_G\rangle$$ 
(here $\inv_G\not\in \Hol(G)$ because $G$ is nonabelian) and therefore
\begin{multline*} N_{\Sym(G)}(H)/H\cong \frac{\Hol(G)\rtimes\langle \inv_G\rangle}{\InHol(G)\rtimes\langle\inv_G\rangle}\cong\Hol(G)/\InHol(G)\cong\frac{\lambda(G)\rtimes\Aut(G)}{\lambda(G)\rtimes\Inn(G)} \cong\Aut(G)/\Inn(G)\\=\Out(G)\cong T.\end{multline*}
This concludes the proof of Theorem \ref{thm: main}.
\end{proof}

\begin{lem}\label{lem: cor of main thm} Assume that $T$ is a normalizer quotient of $S_n$. Then it is also a normalizer quotient of $S_m$ and of $A_m$ for any $m\ge 2n+1$.\end{lem}

\begin{proof} We assume without loss of generality that $n>1$, otherwise the assertion is obvious. We naturally view $S_n\times S_{m-n}$ as a subgroup of $S_m$, letting $S_n$ act on $\{1,\ldots,n\}$ and $S_{m-n}$ on $\{n+1,\ldots,m\}$. Let $H\le S_n$ be such that $N_{S_n}(H)/H\cong T$. Define
$H_1=H\times S_{m-n}$ and $H_2=(H\times S_{m-n})\cap A_m$. We claim that $N_{S_{m}}(H_1)/H_1\cong N_{A_m}(H_2)/H_2\cong T$, which would establish the assertion of the lemma.

Since $m-n>n$, $H_i\le S_n\times S_{m-n}$, and $H_i$ acts transitively on $\{n+1,\ldots,m\}$ (it contains $A_{m-n}$ and $m-n\ge 3$ because $n>1$), the only invariant subset $X\subset\{1,\ldots,m\}$ of size $m-n$ for $H_i$ is $\{n+1,\ldots,m\}$. Therefore if $g\in S_m$ normalizes $H_i$, then $g$ preserves $\{n+1,\ldots,m\}$ and so $g\in S_n\times S_{m-n}$. We conclude that $N_{S_m}(H_i)\le S_n\times S_{m-n}$ and in particular $N_{A_m}(H_2)\le N_{S_m}(H_2)\le S_n\times S_{m-n}$. 

In the case of $H_1=H\times S_{m-n}$ it follows that $$N_{S_m}(H_1)=N_{S_n}(H)\times N_{S_{m-n}}(S_{m-n})=N_{S_n}(H)\times S_{m-n}$$ and therefore $N_{S_m}(H_1)/H_1\cong N_{S_n}(H)/H\cong T$. It remains to show the same for $H_2 = H_1\cap A_m$.

Clearly $N_{S_m}(H_1)\le N_{S_m}(H_2)$ because $A_m\trianglelefteq S_m$, so in particular $(N_{S_n}(H)\times S_{m-n})\cap A_m\le N_{A_m}(H_2)$. We will show the reverse inclusion. Let $g=(a,b)\in N_{A_m}(H_2)$, where $a\in S_n,b\in S_{m-n}$. For any $h\in H$ we can pick $c\in S_{m-n}$ such that $(h,c)\in H_2$. Since $(a,b)$ normalizes $H_2$ we have $a^{-1}ha\in H$ and therefore $a\in N_{S_n}(H)$ (since $h\in H$ can be arbitrary). Thus $g=(a,b)\in (N_{S_n}(H)\times S_{m-n})\cap A_m$. We conclude that $N_{A_m}(H_2)=(N_{S_n}(H)\times S_{m-n})\cap A_m$. 

Since $H\times S_{m-n}$ is not contained in $A_m$, we see that 
$$[\Gamma:\Gamma\cap A_m]=[\Gamma A_m:A_m]=[S_m:A_m]=2$$
for both $\Gamma = N_{S_n}(H)\times S_{m-n},\, H\times S_{m-n}$, and so the injection
$$N_{A_m}(H_2)/H_2= \frac{(N_{S_n}(H)\times S_{m-n})\cap A_m}{(H\times S_{m-n})\cap A_m}\hookrightarrow\frac{N_{S_n}(H)\times S_{m-n}}{H\times S_{m-n}}\cong T$$ is an isomorphism (since both quotients have the same size).
\end{proof}

\begin{proof}[Proof of Corollary \ref{cor: main}] Immediate from Theorem \ref{thm: main} and Lemma \ref{lem: cor of main thm} (take $n_0(T)=2n+1$ where $T$ is a normalizer quotient of $S_n$).\end{proof}

\bibliography{mybib, cindy-bib}
\bibliographystyle{alpha}

\end{document}